[1994/12/01]

 \documentclass{amsart}

\chardef\bslash=`\\ 

\newtheorem{thm}{Theorem}[section]
\newtheorem{cor}[thm]{Corollary}
\newtheorem{lem}[thm]{Lemma}
\newtheorem{prop}[thm]{Proposition}

\theoremstyle{definition}
\newtheorem{defn}{Definition}[section]
\newtheorem{rem}{Remark}[section]

\newtheorem{example}{Example}[section]

\theoremstyle{remark}

\usepackage{graphicx}
\usepackage{xcolor}

\newcommand{\bbHr}{\mathbb{H}^2(r)}

\newcommand{\rmd}{\mathrm{d}}
\newcommand{\rme}{\mathrm{e}}

\newcommand{\eval}[2][\right]{\relax
  \ifx#1\right\relax \left.\fi#2#1\rvert}

\begin{document}
\title {Catenaries in Riemannian Surfaces}
\author{Luiz C. B. da Silva }
\address{Department of Physics of Complex Systems,\\ Weizmann Institute of Science,\\ Rehovot 7610001, Israel}
\email{luiz.da-silva@weizmann.ac.il} 
\author {Rafael L\'opez}
\address{Departamento de Geometr\'ia y Topolog\'ia\\
 Universidad de Granada\\
           18071 Granada, Spain}
           \email{rcamino@ugr.es}

\date{Received on MONTH, YEAR}
\issueinfo{VOL}{NUM}{MONTH}{YEAR}
\begin{abstract}
The concept of catenary has been recently extended to the sphere and the hyperbolic plane by the second author [L\'opez, arXiv:2208.13694].   In this work, we define catenaries on any Riemannian surface. A catenary on a surface is a critical point of the potential functional, where we calculate the potential with the intrinsic distance to a fixed reference geodesic. Adopting semi-geodesic coordinates around the reference geodesic, we characterize catenaries using their curvature. Finally, after revisiting the space-form catenaries, we consider  surfaces of revolution (where a Clairaut relation is established), ruled surfaces, and the Gru\v{s}in plane.
\end{abstract}
\maketitle

\section{Introduction}
The shape of an inextensible heavy chain suspended from its  weight attracted the interest of many scientists, beginning with Galileo. The solution curve of this problem is the catenary, and its derivation is due to  Leibniz, Huygens, and Johann Bernoulli in the 17th century. Very recently, the second author has extended the catenary problem to the sphere $\mathbb{S}^2$ and the hyperbolic plane $\mathbb{H}^2$ \cite{lopezcatenaryform}. In the Euclidean plane, the gravitational potential energy of a curve is calculated using the distance to a given straight line. On the other hand, in $\mathbb{S}^2$ and $\mathbb{H}^2$,    the potential energy is measured with the intrinsic distance to a given geodesic.  Later, extensions to the de Sitter and the simply isotropic spaces have also been investigated \cite{silvalopezisotropic,lopezcatenarysitter}.

This paper aims to unify all these generalizations of the catenary problem by considering an abstract $2$-dimensional Riemannian manifold $(\varSigma^2,\rmd s^2)$.  Given a geodesic $\ell$ of $\varSigma^2$, the {\it catenary problem} consists of finding the shape of a curve $\gamma:[a,b]\to\varSigma^2$ which is a critical point of the energy functional
\begin{equation}\label{f1}
\gamma\longmapsto \int_a^b \mbox{dist}(\gamma(t),\ell)\Vert\dot{\gamma}(t)\Vert\,\rmd t,
\end{equation}
where $\mbox{dist}(\gamma(t),\ell)$ is the intrinsic distance on $\varSigma^2$ between the point $\gamma(t)$ and the fixed geodesic $\ell$. For deriving the critical points of the functional \eqref{f1}, we will use standard techniques of calculus of variations. In particular, we need suitable local coordinates on $\varSigma^2$ that reflect the characteristics of the problem. These are the so-called {\it semi-geodesic} coordinates, which exist around a neighborhood of any geodesic thanks to the exponential map. 

The remaining of this work is divided as follows. In Section \ref{sec2}, the variational problem is presented, and we obtain the Euler-Lagrange equation  of the catenary problem, which provides an expression for the curvature of the solution curve in semi-geodesic coordinates on $\varSigma^2$ (Theorem \ref{t2}). In addition, we remark that the catenary problem can be extended to the context where the reference curve 
$\ell$ is not a geodesic, and apply this theoretical framework to ruled surfaces of Euclidean space $\mathbb{E}^3$. Next, in Section \ref{sec3}, we revisit the catenaries in the space forms $\mathbb{S}^2$ and $\mathbb{H}^2$. Section \ref{sec4} studies the case that $\varSigma^2$ is a surface of revolution in Euclidean space. Surfaces of revolution are highly symmetric, presenting some peculiarities. For instance, we prove a Clairaut-like formula that determines the catenaries in terms of the angle they make with the parallels of the surface (Theorem \ref{tclairaut}). As explicit examples of surfaces of revolution, we calculate the catenaries in circular cones and the catenoid. Finally, in Section \ref{sec5}, we discuss catenaries on the so-called Gru\v{s}in plane.

\section{The catenary problem}\label{sec2}

In this section, we formulate the catenary problem and characterize its solutions. Let $(\varSigma^2,\rmd s^2)$ be a   Riemannian surface and let $\ell:I\to\varSigma^2$ be a fixed geodesic with arc-length parameter $v$. If $\beta_v(u)$ denotes the geodesic emanating from $\ell(v)$ with unit velocity $X(v)$ orthogonal to $\ell(v)$, then there is a neighborhood $\mathcal{U}\subseteq\varSigma^2$ of $\ell(v)$  parametrized by 
\begin{equation*}
\psi(u,v) = \beta_v(u) = \exp_{\ell(v)}(uX(v)).
\end{equation*}
This coordinate system is called {\it semi-geodesic coordinates} \cite{struik}. Note that the geodesic $\ell$ is the parametric curve $v\mapsto \psi(0,v)$. The metric $\rmd s^2$ of  $\varSigma^2$ is then given by 
\begin{equation}\label{G}
    \rmd s^2 = \rmd u^2 + G^2(u,v)\,\rmd v^2,
\end{equation}
where $G$ is a smooth positive function. 

Since $\rmd u\leq \rmd s$,  the distance from a point $\psi(u,v)\in \mathcal{U}$ to $\ell$ is given by $\vert u\vert$. Let $\mathcal{U}_+=\{\psi(u,v)\in \mathcal{U}:u>0\}$. If $\gamma\colon [a,b]\to\mathcal{U}_+$, $\gamma(t)=\psi(u(t),v(t))$,  represents an inextensible heavy chain of constant linear density $\sigma$, the potential energy of $\gamma$ with respect to the reference line $\ell$ is
$$\int_a^b\sigma\, \mbox{dist}(\gamma(t),\ell)\, \rmd s=\int_a^b\sigma\, u \Vert \dot{\gamma}\Vert\, \rmd t.$$
Here, $\dot{\gamma}$ stands for the derivative of $\gamma$ with respect to $t$. From now on, we assume that $\sigma=1$.

We also generalize the hanging chain problem in $\varSigma^2$ by introducing a real parameter $\alpha$ in the potential energy functional. More precisely, we have  

\begin{defn}\label{def::AlphaCatenaryOnSurf}
Let $\alpha\in\mathbb{R}$. A curve $\gamma\colon[a,b]\to \mathcal{U}_{+}$ is said to be  an \emph{$\alpha$-catenary} with respect to $\ell$ if $\gamma$ is a critical point   of the functional
 $$\mathcal{E}[\gamma] = \int_a^b\mbox{dist}(\gamma(t),\ell)^\alpha\, \rmd s.$$ 
\end{defn}

Using semi-geodesic coordinates $(u,v)$, an $\alpha$-catenary  $\gamma(t)=\psi(u(t),v(t))$ is a critical point of the functional
\begin{equation}\label{energy}
    \mathcal{E}[\gamma] =\mathcal{E}[u(t),v(t)]= \int_a^bu^{\alpha} \Vert\dot{\gamma}(t)\Vert\,\rmd t = \int_a^bu^{\alpha}\sqrt{\dot{u}^2+\dot{v}^2G^2(u,v)}\,\rmd t.
\end{equation}
If $\alpha=1$, we simply say that $\gamma$ is a {\it catenary}.

\begin{rem}\label{re1}
Note that $u^\alpha \Vert\dot{\gamma}(t)\Vert\, \rmd t$ is the arc-length element of the conformal metric $\rmd\tilde{s}^2=u^{2\alpha}\rmd s^2$. Thus, critical points of $\mathcal{E}$ coincide with the geodesics of the space $(\varSigma^2, \rmd\tilde{s}^2)$. In particular, this identification guarantees the local existence and uniqueness of $\alpha$-catenaries.
\end{rem}

From now on, all curves will be contained in   $\mathcal{U}_+$. In what follows, it will prove useful to have a suitable expression for the curvature $\kappa$ of a curve in semi-geodesic coordinates. 

\begin{lem}\label{lemma::CurvatureSemiGeoCoord}
 If $\gamma(t)=\psi(u(t),v(t))$ is a curve in $\mathcal{U}_+$ parametrized by semi-geodesic coordinates,  then its (signed) geodesic curvature $\kappa$ is given by
\begin{equation}\label{kk}
    \kappa = -\frac{\dot{v}(G_v\dot{u}\dot{v}+2G_u\dot{u}^2+G^2G_u\dot{v}^2)+G(\dot{u}\ddot{v}-\ddot{u}\dot{v})}{(\dot{u}^2+G^2\dot{v}^2)^{\frac{3}{2}}}.
\end{equation}
\end{lem}
\begin{proof}

Let $\Gamma_{ij}^k$ be the Christoffel symbols associated with the metric $\rmd{s}^2$ given in \eqref{G}. The expression of $\kappa$ in terms of $\Gamma_{ij}^k$  is
\[
    \kappa = \frac{\sqrt{g_{ij}}\left(\Gamma_{22}^1\dot{v}^3-\Gamma_{11}^2\dot{u}^3-(2\Gamma_{12}^2-\Gamma_{11}^1)\dot{u}^2\dot{v}+(2\Gamma_{12}^1-\Gamma_{22}^2)\dot{u}\dot{v}^2+\ddot{u}\dot{v}-\dot{u}\ddot{v}\right)}{(g_{11}\dot{u}^2+2g_{12}\dot{u}\dot{u}+g_{22}\dot{v}^2)^{3/2}}.
\]
For the metric $\rmd{s}^2$, we have $g_{11}=1$, $g_{12}=0$, and $g_{22}=G^{2}$. Using the well-known expression for the Christoffel symbols, $\Gamma_{ij}^k=\frac{g^{k\ell}}{2}(\partial_ig_{j\ell}+\partial_jg_{i\ell}-\partial_{\ell}g_{ij})$, where $u^1=u$ and $u^2=v$, we obtain $\Gamma_{11}^1 = \Gamma_{12}^1 =  \Gamma_{11}^2 = 0$, and 
\begin{equation}\label{co}
  \Gamma_{22}^1 = -GG_u,  \quad \Gamma_{12}^2 = \frac{G_u}{G},  \quad \Gamma_{22}^2 =\frac{G_v}{G}.
\end{equation}
Finally, substituting the values of $\Gamma_{ij}^k$ in the expression for $\kappa$ proves Eq. \eqref{kk}.
\end{proof}
 

\begin{thm}\label{t1}
 Let $\gamma(t)=\psi(u(t),v(t))$ be a smooth curve  parametrized by semi-geodesic coordinates. Then, $\gamma$ is an $\alpha$-catenary if, and only if, it satisfies
 \begin{equation}\label{el0}
  \dot{v}\, u^\alpha G\left(\kappa-\alpha\frac{G\dot{v}}{u \Vert\dot{\gamma}\Vert}\right) = 0 \quad \mbox{and} \quad
   \dot{u}\,u^\alpha G\left(\kappa-\alpha\frac{G\dot{v}}{u \Vert\dot{\gamma}\Vert}\right)  = 0,
\end{equation}
where  $\kappa$ is the geodesic curvature of   $\gamma$. In addition, if   $\gamma$ is regular, then $\gamma$ is an $\alpha$-catenary if, and only if,  
\begin{equation}\label{k-c}
    \kappa=\alpha\frac{G\dot{v}}{u \Vert\dot{\gamma}\Vert}=\alpha\frac{G\dot{v}}{u\sqrt{\dot{u}^2+G^2\dot{v}^2}}.
\end{equation}
\end{thm}
\begin{proof}
Since the Lagrangian associated with $\mathcal{E}$ is $L=u^\alpha \Vert\dot{\gamma}\Vert$, the corresponding Euler-Lagrange equations   are
\begin{equation}\label{el1}
\frac{\partial L}{\partial u}-\frac{\rmd}{\rmd t}\frac{\partial L}{\partial \dot{u}} = 0 \quad \mbox{and} \quad \frac{\partial L}{\partial v}-\frac{\rmd}{\rmd t}\frac{\partial L}{\partial \dot{v}} = 0.
\end{equation}

For the first equation, we need
\begin{equation*}
    \frac{\partial L}{\partial u} =\alpha u^{\alpha-1}\, \Vert\dot{\gamma}\Vert+u^\alpha\,\frac{GG_u\dot{v}^2}{ \Vert\dot{\gamma}\Vert}
\quad \mbox{and} \quad  \frac{\partial L}{\partial \dot{u}} = u^\alpha\frac{\dot{u}}{ \Vert\dot{\gamma}\Vert}.
\end{equation*}
The  derivative of $\partial L/\partial \dot{u}$ with respect to $t$  is
\begin{eqnarray*}
    \frac{\rmd}{\rmd t}\frac{\partial L}{\partial \dot{u}}  
    & = & \frac{\alpha u^{\alpha-1}\dot{u}^2}{ \Vert\dot{\gamma}\Vert}+\frac{u^\alpha}{ \Vert\dot{\gamma}\Vert^2}\left(\ddot{u} \Vert\dot{\gamma}\Vert-\dot{u}\frac{\dot{u}\ddot{u}+G(\dot{u}G_u+\dot{v}G_v)\dot{v}^2+G^2\dot{v}\ddot{v}}{ \Vert\dot{\gamma}\Vert}\right) \nonumber \\
        & = & \frac{\alpha u^{\alpha-1}\dot{u}^2}{ \Vert\dot{\gamma}\Vert}+\frac{\dot{v}u^\alpha G}{ \Vert\dot{\gamma}\Vert^3}[G(\dot{v}\ddot{u}-\dot{u}\ddot{v})-\dot{u}\dot{v}(G_u\dot{u}+G_v\dot{v})].
\end{eqnarray*}
Thus, the first Euler-Lagrange equation  becomes
\begin{eqnarray*} 
    0     & = & \frac{\alpha u^{\alpha-1}G^2\dot{v}^2+u^\alpha GG_u\dot{v}^2}{ \Vert\dot{\gamma}\Vert}  -\dot{v} u^\alpha G\frac{G(\dot{v}\ddot{u}-\dot{u}\ddot{v})-\dot{u}\dot{v}(G_u\dot{u}+G_v\dot{v})}{ \Vert\dot{\gamma}\Vert^3} \nonumber \\
         & = & \frac{\alpha u^{\alpha-1}G^2\dot{v}^2}{ \Vert\dot{\gamma}\Vert} -  \dot{v} u^\alpha G\frac{\dot{v}(G\ddot{u}-G_v\dot{u}\dot{v}-2G_u\dot{u}^2-G^2G_u\dot{v}^2)-G\dot{u}\ddot{v}}{ \Vert\dot{\gamma}\Vert^3}.
\end{eqnarray*}
Replacing the value of $\kappa$ given in \eqref{kk} in the last part of the above identity,  we can write
\begin{equation}
    0 =  \frac{\alpha u^{\alpha-1}G^2\dot{v}^2}{ \Vert\dot{\gamma}\Vert} -  \dot{v}u^\alpha G\kappa = -\dot{v} u^\alpha G\left(\kappa-\alpha  \frac{G\dot{v}}{u \Vert\dot{\gamma}\Vert}\right).
\end{equation}
This is the first equation of \eqref{el0}. For the second Euler-Lagrange equation, we may proceed analogously.  
\end{proof}

We can alternatively characterize $\alpha$-catenaries without using the curvature $\kappa$ thanks to Eq.  \eqref{kk}. This  will prove to be useful in explicit calculations. We have

\begin{cor}\label{c2} Let $\ell:I\to\varSigma^2$ be a geodesic. A regular curve  $\gamma(t)=\psi(u(t),v(t))$  parametrized by the semi-geodesic coordinates  is an $\alpha$-catenary with respect to $\ell$ if, and only if,
\begin{equation}\label{et3}
\frac{\alpha\dot{v}G}{u}\Vert\dot{\gamma}\Vert^2=-\dot{v}( G_v\dot{u}\dot{v}+2G_u\dot{u}^2+G^2G_u\dot{v}^2)-G(\dot{u}\ddot{v}-\ddot{u}\dot{v}).
\end{equation}
\end{cor}

\begin{proof} It is a consequence of Eqs. \eqref{kk} and \eqref{k-c}.
\end{proof}


We now provide a coordinate-free characterization for $\alpha$-catenaries. Consider the vector field $X=\partial_u$ that corresponds to the velocity field of the geodesics with length equal to $\mbox{dist}(\gamma(t),\ell)=u$. (If $G(u,v)=G(v)$, then the translations $(u,v)\mapsto (u+\mu,v)$ are isometries of $(\varSigma,\rmd s^2)$ and, consequently, $X$ is a Killing vector field.) We now obtain the following characterization of the $\alpha$-catenaries involving the curvature $\kappa$ of $\gamma$ and the geodesic velocity vector field $X$.

\begin{thm}\label{t2}
Let $\ell:I\to\varSigma^2$ be a geodesic. A regular curve  $\gamma(t)=\psi(u(t),v(t))$  parametrized by semi-geodesic coordinates  is an $\alpha$-catenary with respect to $\ell$ if, and only if, its curvature $\kappa$ satisfies  
\begin{equation}\label{free}
    \kappa = \alpha\frac{\langle\mathbf{n}, X\rangle}{\mathrm{dist}(\gamma,\ell)},
\end{equation}
where $\mathbf{n}$ is the unit normal vector to $\gamma$.
\end{thm}
\begin{proof}
The normal vector $\mathbf{n}$ is given by 
\begin{equation}\label{pt1}
 \mathbf{n}=\frac{1}{\Vert\dot{\gamma}\Vert}\left(-\dot{v}G\frac{\partial}{\partial u}+\frac{\dot{u}}{G}\frac{\partial}{\partial v}\right).
\end{equation}
Thus, $\langle\mathbf{n},X\rangle=-\dot{v}G/\Vert\dot{\gamma}\Vert$,  and the result follows immediately. 
\end{proof}

\begin{rem} 
The motivation of the catenary problem, i.e., $\alpha=1$, is finding a hanging chain suspended from its ends. If the physical chain is inextensible, its length is prescribed. In the variational formulation of the catenaries, a Lagrange multiplier must be considered due to the condition $\int_a^b\Vert\dot{\gamma}\Vert\, dt=c$. Thus, the functional $\mathcal{E}$ is replaced by $\int_a^b(u+\lambda)\Vert\dot{\gamma}\Vert\, dt$. Repeating the computations, a regular curve $\gamma$ is a critical point of this functional if, and only if, its curvature satisfies 
 \begin{equation*} 
    \kappa=\alpha\frac{G\dot{v}}{(u+\lambda) \Vert\dot{\gamma}\Vert}.
\end{equation*}
Unless $G(u,v)=G(v)$, translations of the type $(u,v)\mapsto (u+\mu,v)$ are not isometries of $(\Sigma,\rmd s^2)$. Thus, the introduction of a Lagrange multiplier leads to a distinct problem.
\end{rem}

\subsection{Catenaries with respect to a non-geodesic reference curve}

The validity of Theorems    \ref{t1} and   \ref{t2} does not depend on the fact that the curve $v\mapsto\psi(0,v)$ is a geodesic. Indeed, analyzing the proofs of all the results obtained so far  reveals that their validity relies solely on the special form  of the metric in a neighborhood $\mathcal{U}$ of the curve $v\mapsto\psi(0,v)$, Eq. \eqref{G}. Consequently, in this paper, we shall drop the assumption that $\ell$ is a geodesic. Thus, we may define

\begin{defn}\label{Def::alpCatenaryRespectToNonGeod} 
Let $(\varSigma^2,\rmd s^2)$ be a Riemannian surface parametrized by $\psi(u,v)$ and with metric $\rmd s^2=\rmd u^2+G^2(u,v)\rmd v^2$. Defining $\ell(v)=\psi(0,v)$, we say that $\gamma\colon [a,b]\to \mathcal{U}_+$ is an \emph{$\alpha$-catenary with respect to $\ell$} if $\gamma$ is a critical point of the  functional \eqref{energy}, where $\mathcal{U}_+=\{\psi(u,v)\in\Sigma^2: u>0\}$ and $\alpha\in\mathbb{R}$.
\end{defn}

Consequently, Theorems \ref{t1} and \ref{t2} and Corollary \ref{c2} are valid for this definition of $\alpha$-catenary. 

The characterization of $\alpha$-catenaries in Theorem \ref{t1} only involves the metric of the surface. This is also expected by Remark \ref{re1}. Thus, the concept of $\alpha$-catenary is preserved by local isometries. So, the following result is immediate.

\begin{prop}\label{iso}
    If two   surfaces $\varSigma^2_1$ and $\varSigma^2_2$  are parametrized  by the same semi-geodesic coordinates, then the $\alpha$-catenaries in both surfaces coincide.
\end{prop}
    
We finish this section investigating when the coordinate curves in the semi-geodesic coordinates systems are $\alpha$-catenaries. It is immediate from Corollary \ref{c2} the following result.

\begin{cor}\label{c3} Let $(\varSigma^2,\rmd s^2)$ be a   Riemannian surface parametrized by $\psi(u,v)$, where  $\rmd s^2=\rmd u^2+G^2(u,v)\rmd v^2$.  Let $\ell(v)=\psi(0,v)$. Then:
\begin{enumerate}
\item Any coordinate curve $u\mapsto \psi(u,v_0)$ is an $\alpha$-catenary.
\item  A coordinate curve $v\mapsto \psi(u_0,v)$ is an $\alpha$-catenary if, and only if,  the equation $\alpha G+u_0 G_u=0$ is valid along the points of the curve.
 \end{enumerate}
It follows that, if an $\alpha$-catenary $\gamma$  is tangent to some coordinate curve $v=v_0$, then $\gamma$ must be the coordinate curve $u\mapsto \psi(u,v_0)$. (See Remark \ref{re1}.) 
\end{cor}

As a consequence of the last statement of Corollary \ref{c3}, if we exclude the trivial case of the coordinate curves $v=v_0$, any $\alpha$-catenary is a graph over any coordinate curve $u=u_0$, such as the reference curve $\ell$. Thus, to calculate $\alpha$-catenaries on a given surface, we can apply Eq. \eqref{et3} with $u=u(v)$, which then implies we must solve
\[
\frac{\alpha G}{u}\Vert\dot{\gamma}\Vert^2=-( G_v\dot{u}+2G_u\dot{u}^2+G^2G_u)+G\ddot{u}.
\]

\begin{example}[Ruled surfaces in Euclidean space]
Let $c:I\to\mathbb{E}^3$ be a regular curve and $W$ a unit vector field along $c$. Consider the ruled surface $\varSigma^2$ parametrized by
  \begin{equation}\label{ruled}
    \psi(u,v)=c(v)+uW(v),\quad v\in I, u\in\mathbb{R}.
\end{equation}
Without loss of generality, we may assume that $\langle c',W\rangle=0$ and that $c$ is parametrized by arc length. Considering the reference curve $\ell$ to be $c(u)=\psi(0,v)$, the geodesics of $\varSigma^2$ orthogonal to $c(v)$ are the rulings $u\mapsto \psi(u,v)$. The metric of $\Sigma^2$ is $\rmd u^2+G^2\rmd v^2$, where 
$$
G(u,v)^2=1+2u\langle c'(v),W'(v)\rangle+u^2\Vert W'(v)\Vert^2.
$$
We present catenaries in explicit examples of ruled surfaces. 
\begin{enumerate}
    \item \emph{Cylindrical surfaces}. A cylindrical surface is a ruled surface with  $c(v)$ as a plane curve in $\mathbb{E}^3$   and   $W$ as a unit constant vector orthogonal to the plane containing $c(v)$.  In this case, we have  $G=1$; therefore, the $\alpha$-catenaries in cylindrical surfaces coincide with that of the Euclidean plane (see Proposition \ref{iso}). 
    \item \emph{Helicoid}. The helicoid is the ruled surface with  $c(v)=(0,0,v)$ and   $W(v)=(\cos v,\sin v,0)$. The function $G$ is   $G(u,v) =\sqrt{1+u^2}$. Here, the reference curve $\ell(v)=c(v)=\psi(0,v)$ is also a geodesic of the helicoid. Thus, the equation of the catenaries \eqref{kc5} becomes 
    $$
      u(1+u^2)\ddot{u}-[2u^2+\alpha(1+u^2)]\dot{u}^2-(1+2\alpha)u^2-(1+\alpha)u^4-\alpha=0.
    $$
    In the particular case $\alpha=1$, this equation is
    \[
     u(1+u^2) \ddot{u}-(1+3 u^2) \dot{u}^2-2 u^4-3u^2-1=0.    
    \]
    \item \emph{Binormal surface.} Let $c:I\to\mathbb{E}^3$ be a regular curve with curvature $\bar{\kappa}$ and torsion $\bar{\tau}$. If $\{T,N,B\}$ denotes the Frenet frame of $c$, the binormal surface of $c$ is the ruled surface \eqref{ruled} with base curve  $c$ and $W=B$, i.e., the rulings are the binormal lines. The function $G$ is $G=\sqrt{1+\bar{\tau}^2 u^2}$. If the torsion $\bar{\tau}\not=0$ is constant, then $\Sigma^2$ is locally isometric to a helicoid. 
\end{enumerate}
\end{example}
 
\section{Catenaries in space forms}
\label{sec3}
 
The first context where we apply the characterization of catenaries provided by Theorem \ref{t1} is that of space forms. The $2d$ space forms $M_k^2$ of curvature $k$ can be viewed as a warped product on $[0,\Lambda)\times \mathbb{S}^1$ with warped metric $\rmd s^2= \rmd u^2+s_k^2(u)\rmd \theta^2$, where $\rmd \theta^2$ is the standard metric on the unit circle $\mathbb{S}^1$, 
$$
 s_k(u) = \left\{
          \begin{array}{cl}
            \dfrac{\sin(\sqrt{k}\,u)}{\sqrt{k}},&k>0\\
            u,& k=0\\
            \dfrac{\sinh(\sqrt{-k}\,u)}{\sqrt{-k}},&k<0
          \end{array}
          \right., \quad \mbox{and} \quad
 \Lambda =\left\{
          \begin{array}{cl}
            \frac{\pi}{\sqrt{k}},&k>0\\
            \infty,& k\leq0
          \end{array}
          \right..
$$ 
The standard model for two-dimensional space forms is 
$$
 M_k^2 = \left\{
         \begin{array}{lll}
          \mbox{Sphere $\mathbb{S}^2(r)$ of radius $r$}, &  k=\frac{1}{r^2}\\
          \mbox{Euclidean plane $\mathbb{E}^2$}, & k=0\\
          \mbox{Hyperbolic plane $\mathbb{H}^2(r)$}, & k=-\frac{1}{r^2} 
         \end{array}
         \right..
$$

\begin{example}[The Euclidean plane] Let $\varSigma^2=\mathbb{E}^2$ be the Euclidean plane. If the metric is $\rmd s^2=\rmd u^2+\rmd v^2$, consider the geodesic $\ell$ given by the equation $u=0$. If $\gamma$ is parametrized by $\gamma(t)=(u(t),t)$, then  \eqref{et3} is 
\begin{equation}\label{acatenary}
\frac{\alpha}{u}=\frac{\ddot{u}}{1+\dot{u}^2}.
\end{equation}
For $\alpha=1$, we obtain the well-known Euclidean catenary (see Fig. \ref{fig:CylindricalCatenaries}, left):
\begin{equation}\label{rcate2}
u(t)=\frac{1}{\mu}\cosh (\mu t+\nu),\quad \mu,\nu\in\mathbb{R} \quad (\mu\not=0).
\end{equation}

It is worth mentioning that Eq. \eqref{acatenary} also appears in the theory of singular minimal surfaces. Indeed, the solutions of this equation are the generating curves of cylindrical singular minimal surfaces \cite{dierkesgroh,dierkeslopez,lopezinvariant}. Multiplying both sides of \eqref{acatenary} by $\dot{u}$ and then integrating leads to $1+\dot{u}^2=\mu u^{2\alpha}$, for some constant $\mu>0$.  The geometry of these solutions is described in \cite{dierkesgroh,dierkeslopez,lopezinvariant}.  
\end{example}
 
 \begin{example}[The sphere] 
 The curvature $\kappa$ of a spherical $\alpha$-catenary satisfies
$$
 \kappa=\frac{\alpha\cos{u}}{u\sqrt{\dot{u}^2+(\cos{u})^2}}.
$$
This equation coincides with the results in Ref. \cite{lopezcatenaryform}. We shall provide more information on spherical catenaries in Section \ref{sec4} when discussing catenaries on Euclidean surfaces of revolution.
 
 \end{example}
 
 \begin{example}[The hyperbolic plane] Let $\mathbb{H}^2(r)$ be the hyperbolic plane in the hyperboloid model, i.e., consider $\mathbb{H}^2(r)$ as the  surface of curvature $-1/r$ in the $3d$ Lorentzian space  $\mathbb{E}_1^{3}=(\mathbb{R}^{3},\langle X,Y\rangle_1=-X_0Y_0+X_1Y_1+ X_2Y_2)$ given by 
\begin{equation}
    \mathbb{H}^2(r) = \{(x,y,z)\in\mathbb{E}_1^{3}:-x^2+y^2+z^2=-r^2,\,x>0\}.
\end{equation}
Let $\ell$ be the geodesic in $\bbHr$ obtained by the intersection with the plane of equation $z=0$:
\begin{equation}
    \ell(v) = r(\cosh\frac{v}{r},\sinh\frac{v}{r},0).
\end{equation}
The geodesics orthogonal to $\ell$ at $\ell(v)$ has velocity vector $X=(0,0,1)$. Thus, we parametrize $\bbHr$ as
\begin{equation}\label{eq35}
  \psi(u,v) = \ell(v)\cosh \frac{u}{r}+\sinh \frac{u}{r}\,X= r(\cosh\frac{u}{r}\cosh\frac{v}{r},\cosh\frac{u}{r}\sinh\frac{v}{r},\sinh\frac{u}{r}).
\end{equation} 
A direct computation shows that the   induced metric takes the form
\begin{equation}
    \rmd s^2 = \rmd u^2+\cosh^2 \frac{u}{r}\,\rmd v^2.
\end{equation}
Thus,  $G(u,v)=\cosh\frac{u}{r}$. It follows that $\gamma(t)=\psi(u(t),v(t))$ is a hyperbolic $\alpha$-catenary if, and only if, $\gamma$ has curvature
\begin{equation}
    \kappa(t) = \alpha\frac{\dot{v}(t)}{u(t)\Vert\dot{\gamma}(t)\Vert}\cosh \frac{u(t)}{r} .
\end{equation}

Since  the Lagrangian $L=u^\alpha\sqrt{\dot{u}^2+G^2\dot{v}^2}$ does not depend on $v$, it follows that $\alpha$-catenaries have the first integral
\begin{equation}
    \frac{\partial L}{\partial\dot{v}} = u^{\alpha}\frac{\dot{v}\cosh^2\frac{u}{r}}{\sqrt{\dot{u}^2+\dot{v}^2\cosh^2\frac{u}{r}}}=c=\mbox{constant}.
\end{equation}
Therefore, we can integrate the equation for the $\alpha$-catenary by quadrature. Assuming $\gamma$ is a graph over $\ell$, i.e., $u=u(v)$ over an interval $[a,b]$, we have 
\begin{equation}
     c\,\frac{\rmd u}{\rmd v} = \pm\cosh\frac{u}{r}\sqrt{u^{2\alpha}\cosh^2\frac{u}{r}-c^2}\,,
\end{equation}
from which we can find $u$ as a function of $v$ by inverting 
\begin{equation}
    \frac{v}{c} = \pm\int_{u(a)}^{u(b)}\frac{\rmd u}{\cosh\frac{u}{r}\sqrt{u^{2\alpha}\cosh^2\frac{u}{r}-c^2}}.
\end{equation}
 \end{example}

\section{Catenaries on surfaces of revolution in Euclidean space}
\label{sec4}

In this section, we calculate the catenaries on Euclidean surfaces of revolution endowed with the metric induced by $\mathbb{E}^3$. Without loss of generality, we can assume that the rotation axis is the $z$-axis. If the generating curve is  $c(u)=(a(u),0,b(u))$, $a>0$, $u\in I$, $0\in I$, where $u$ is the arc-length parameter, then  $\varSigma^2$ is parametrized by 
\begin{equation}\label{ParametrizationSurfRevolution}
\psi(u,v)=(a(u)\cos{v},a(u)\sin{v},b(u)),\quad u\in I, v\in\mathbb{R}.
\end{equation} 
Then, the metric of $\varSigma^2$ is $\rmd u^2+a(u)^2\rmd v^2$. We shall take as the reference curve $\ell$ the parallel of equation $u=0$, which is a geodesic if, and only if, $a'(0)=0$.  

Applying Corollary \ref{c2}, a curve $\gamma(t)=\psi(u(t),v(t))$ is an $\alpha$-catenary if, and only if, it satisfies  
\[
\alpha a \dot{v} (\dot{u}^2+a^2\dot{v}^2) = -u\left[\dot{v}( 2a'\dot{u}^2+a^2a'\dot{v}^2)+a(\dot{u}\ddot{v}-\ddot{u}\dot{v})\right].
\]
In addition, from Corollary \ref{c3}, the meridians of the surface are catenaries. As discussed after Corollary \ref{c3},  if $\gamma$ is not a meridian, we may write $u=u(v)$.  Then, $\gamma$ is an $\alpha$-catenary  if, and only if, 
\begin{equation}\label{kc5}
\alpha a  (\dot{u}^2+a^2)=u(a\ddot{u}-2a'\dot{u}^2-a^2a').
\end{equation}

As mentioned in Remark \ref{re1},   $\alpha$-catenaries can be seen as geodesics of the conformal metric $\rmd\tilde{s}^2=u^{2\alpha}(\rmd u^2+G^2\rmd v^2)$. Since the  metric $\rmd\tilde{s}^2$ is invariant, it is expected that $\alpha$-catenaries will obey a Clairaut relation \cite{MO11}. Indeed,

\begin{thm}[Clairaut relation for $\alpha$-catenaries]\label{tclairaut}
Let $\gamma$ be an $\alpha$-catenary on a surface of revolution $\Sigma^2$ parametrized by \eqref{ParametrizationSurfRevolution} and which is not a meridian. If $\theta$ denotes the angle between $\gamma$ and the parallels of $\Sigma^2$, then there exists a constant $c\in\mathbb{R}$ such that
\begin{equation}\label{clairautRelation}
    a(u)u^{\alpha}\cos\theta = c.
\end{equation}
We shall refer to $\rho(u)=a(u)u^{\alpha}$ as the \emph{Clairaut radius} of $\Sigma^2$.
\end{thm}

\begin{proof}
Since the Lagrangian $L(u,v,\dot{u},\dot{v})=u^{\alpha}\sqrt{\dot{u}^2+a^2\dot{v}^2}$ associated with an $\alpha$-catenary on a surface of revolution does not depend on $v$, $\partial L/\partial\dot{v}$ is a first integral. In other words, there exists a constant $c\in\mathbb{R}$ such that
\begin{equation}\label{inte1}
   \frac{\partial L}{\partial\dot{v}} = \frac{u^{\alpha}a^2\dot{v}}{\Vert\dot{\gamma}\Vert}=c.
\end{equation}

On the other hand, the parallels of $\varSigma^2$ are the coordinate curves $v\mapsto\beta_u(v)=\psi(u,v)$. Thus, the angle $\theta$ between $\gamma$ and the parallels satisfies 
$\langle\beta_u',\dot{\gamma}\rangle=\Vert \beta_u'\Vert  \Vert \dot{\gamma}\Vert \cos\theta$. Since 
$\langle \beta_u',\dot{\gamma}\rangle=\langle\psi_v,\dot{\gamma}\rangle=\dot{v}a^2$ and $\Vert \beta_u'\Vert =a$, we have 
$\dot{v}a=\Vert \dot{\gamma}\Vert \cos\theta$.   Now, using Eq. \eqref{inte1}, we finally obtain
\[
 a u^{\alpha} \cos\theta = \frac{u^{\alpha}a^2\dot{v}}{\Vert\dot{\gamma}\Vert} = c.
\]\end{proof}

A parallel of a surface of revolution is a geodesic if, and only if, $a'=0$. A similar characterization applies to $\alpha$-catenaries if we replace $a$ with the Clairaut radius $\rho$. Indeed, by Corollary \ref{c3}, a parallel $u=u_0$ is an $\alpha$-catenary if, and only if, $\alpha a+u_0a'=0$, which is equivalent to $\rho'=0$. Therefore, a parallel $u=u_0$ is an $\alpha$-catenary if, and only if, $\rho'(u_0)=0$. 

\begin{cor}
Let $\gamma(v)=\psi(u(v),v)$ be an $\alpha$-catenary on a surface of revolution $\varSigma^2$ parametrized by \eqref{ParametrizationSurfRevolution}. Then, there is a constant $c\in\mathbb{R}$ such that
\begin{equation}\label{SolutionByQuadratureSurfRev}
 v-v_0 = \pm\,\int_{u_0}^u \frac{\rmd t}{\sqrt{\frac{t^{2\alpha}a^2(t)}{c^2}-1}}.
\end{equation}
\end{cor}

\begin{proof}
Assuming that $u=u(v)$ in Eq. \eqref{inte1}, it follows that 
$$ \frac{u^{\alpha}a^2}{\sqrt{\dot{u}^2+a^2}}=c,$$
or, equivalently, 
$$ \dot{u} = \pm\, a\,\sqrt{\frac{u^{2\alpha}a^2}{c^2}-1}.$$
Now, Eq. \eqref{SolutionByQuadratureSurfRev} follows directly.  
\end{proof}

\begin{rem}
Obtaining $u(v)$ from Eq. \eqref{kc5} by quadrature is also possible. In fact, doing $w(u)=\dot{u}(v)$, and using that $ww'=\ddot{u}$, then Eq. \eqref{kc5} writes as
\[
 \alpha a (w^2+a^2) = ua ww'-2ua'w^2-ua'a^2. 
\]
This expression implies 
\[
 \left(\frac{w^2}{2}\right)' =2\left(\frac{\alpha}{u}+\frac{2a'}{a}\right)\frac{w^2}{2} + a^2\left(\frac{\alpha}{u}+\frac{a'}{a}\right),
\]
which is an equation of type $y'=A(u)y+B(u)$, $y=y(u)$, and whose general solution is $y=\mathrm{e}^{\int A}\int B \mathrm{e}^{-\int A}$. In the present case, 
the solution is 
$$\frac{w^2}{2} =\frac{a^2}{2}\left(c_1u^{2\alpha}a^2-1\right),$$
for some integration constant $c_1$. Extracting the square root gives the desired result for $\dot{u}$ with $c_1=c^{-2}$. 
\end{rem}

\begin{figure}[t]
    \centering
    \includegraphics[width=0.4\linewidth]{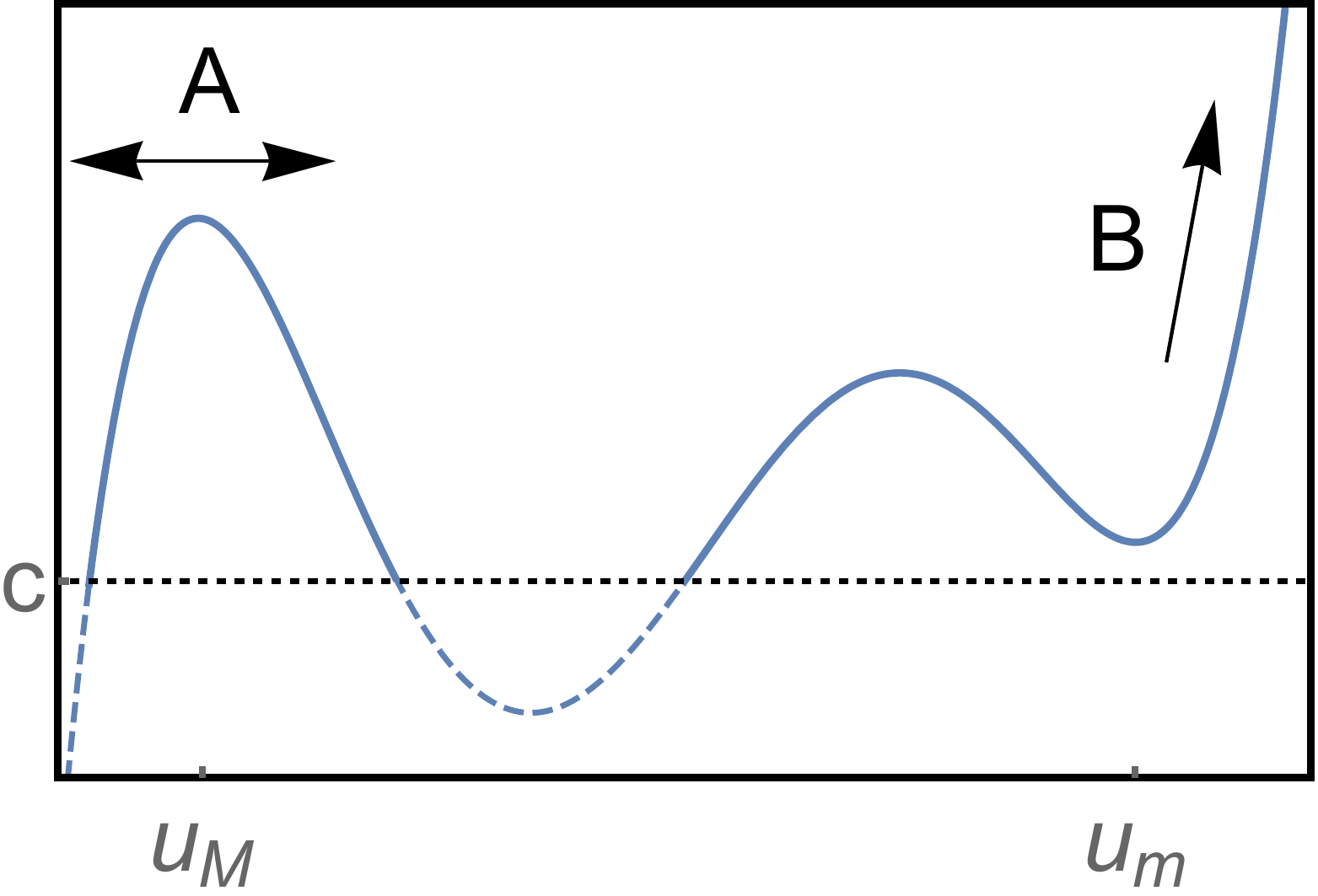}
    \caption{The Clairaut radius $\rho$ of $\alpha$-catenaries on a surface of revolution, Eq. \eqref{clairautRelation}. Once $c$ is fixed, the region $\rho(u)\leq c$ is inaccessible (dashed blue curve). A maximum of $\rho$ is expected to be stable. In the figure, the $\alpha$-catenaries close to the critical parallel $u=u_M$ will be trapped in the region $A$. On the other hand, a minimum of $\rho$ is expected to be unstable. In the figure, the $\alpha$-catenaries close to the critical parallel $u=u_m$ may escape through the region $B$ and never return.}
    \label{fig:GraphClairautRadius}
\end{figure}

Once we fix the constant $c$ in the Clairaut relation, the coordinate $u$ of an $\alpha$-catenary only assumes those values belonging to the set $\{u:\rho(u)\geq c\}$. Then, for $c$ slightly below a maximum of the Clairaut radius $\rho$, the $\alpha$-catenaries are trapped between the two parallels corresponding to $\rho(u)=c$. This situation is depicted by the maximum $u_M$ in Figure \ref{fig:GraphClairautRadius}. In this sense, we may refer to a critical parallel corresponding to a maximum of $\rho$ as a \emph{stable $\alpha$-catenary}. On the other hand, if $u^*$ is a minimum of $\rho$, then an $\alpha$-catenary initially close to the parallel $u=u^*$ may never return. This situation is depicted by the minimum $u_m$ in Figure \ref{fig:GraphClairautRadius}.
 
If we forget for a moment the function multiplying the square root in Eq. \eqref{SolutionByQuadratureSurfRev}, then we could establish a mechanical analogy and see the catenary equation as the equation of motion of a particle of mass $m=2$ subject to a potential $V=-\frac{1}{c^2}\rho^2$ and whose total energy is $E=-1$. Since $V'(u^*)=-\frac{2}{c^2}\rho(u^*)\rho'(u^*)=0$ and $V''(u^*)=-\frac{2}{c^2}\rho(u^*)\rho''(u^*)$, we could explain the intuition obtained from Figure \ref{fig:GraphClairautRadius}. It turns out that such a mechanical analogy can be made precise after we employ a suitable change of coordinates. Indeed, define $u=f(z)$, where $f$ is a solution of the first-order differential equation $y'=a(y)$. Then, we have $\dot{u}=\dot{z}f'(z)$ and, from Eq. \eqref{SolutionByQuadratureSurfRev}, it follows that
\begin{equation}\label{SolByQuadratConformalCoordSurfRev}
    \dot{z} = \frac{\dot{u}}{f'(z)} = \pm\,\sqrt{\frac{\bar{\rho}(z)^2}{c^2}-1},
\end{equation}
where $\bar{\rho}(z)=f(z)^{\alpha}a(f(z))$. Geometrically, the new coordinate system $(z,v)$ is conformally flat:
\[
 \rmd s^2 = \rmd u^2+a(u)^2\rmd v^2 = f'(z)^2\rmd z^2 + a(f(z))^2\rmd v^2 = a(f(z))^2(\rmd z^2 + \rmd v^2).
\]
Note that, in the conformal coordinate system $(z,v)$, the catenary equation \eqref{kc5} becomes
\[
 \bar{\rho}\ddot{z} = \bar{\rho}'\dot{z}^2 + \bar{\rho}'.
\]
From this equation, we can obtain Eq. \eqref{SolByQuadratConformalCoordSurfRev}.

In addition, the conformal coordinate system $(z,v)$ allows us to analyze the stability of $\alpha$-catenaries. Taylor expanding $V$ around a value $z^*$ such that $V'(z^*)=0$, the linearization of the equation of motion $-\ddot{z}=V'(z)$ becomes
\[
 \Delta\ddot{z} = -\lambda\,\Delta z,\quad \Delta z = z-z^*,\quad \lambda = \frac{V''(z^*)}{c^2}.
\]
If $\lambda>0$, the solution is $\Delta z = z_1\cos(\sqrt{\lambda}\,v) + z_2\sin(\sqrt{\lambda}\,v)$ and, consequently, the parallel associated with $z=z^*$ must be (linearly) stable. On the other hand, if $\lambda<0$, the solution is $\Delta z = z_1\cosh(\sqrt{-\lambda}\,v) + z_2\sinh(\sqrt{-\lambda}\,v)$ and, consequently, the parallel associated with $z=z^*$ is not (linearly) stable.
 
\begin{figure}[hbtp]
    \centering
    \includegraphics[width=0.98\linewidth]{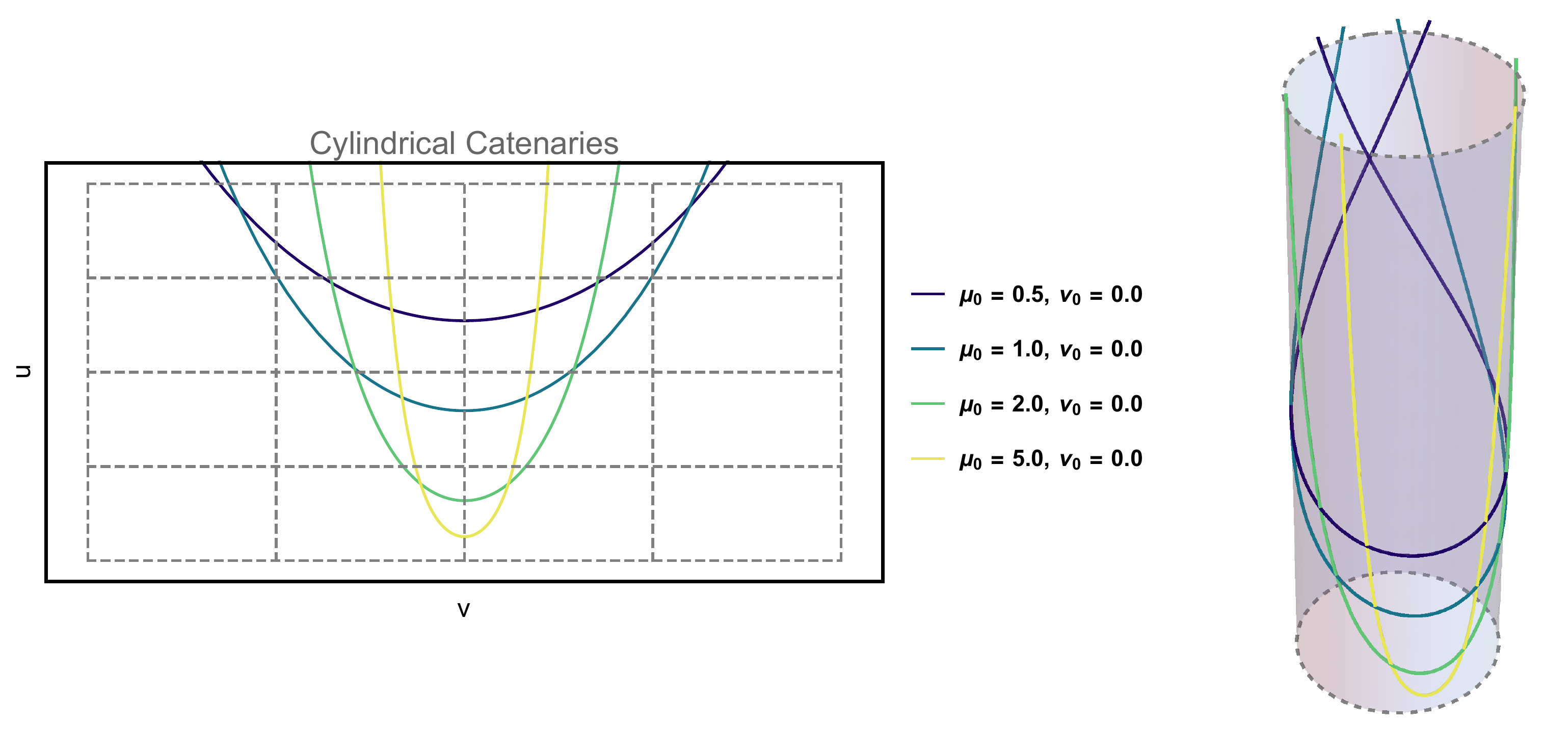}
    \caption{The catenaries $\gamma(v)=(\cos v,\sin v,\frac{1}{\mu_0}\cosh(\mu_0v+\nu_0))$ on a cylinder. (Left) Catenaries on the coordinate $(u,v)$-plane. The dashed gray lines represent the geodesics given by the coordinate curves, while the full curves represent catenaries. (Right) Three-dimensional view of the catenaries on a cylinder.}
    \label{fig:CylindricalCatenaries}
\end{figure}

Now, we provide  examples of catenaries on explicit surfaces of revolution.

\subsection{Circular cylinder}

If $\varSigma^2$ is a circular cylinder of radius $1$ about the $z$-axis, then the generating curve is $c(u)=(1,0,u)$.   Thus, $a(u)=1$ and  no parallel of a cylinder is an $\alpha$-catenary. Since $G=1$, the $\alpha$-catenaries coincide with that of the Euclidean plane.  See Fig. \ref{fig:CylindricalCatenaries}, right.  Note that the Clairaut constant is $c=1/\sqrt{\mu}$.

\subsection{Unit sphere $\mathbb{S}^2$}

The generating curve of $\mathbb{S}^2$ is $c(u)= (\cos{u},0,\sin{u})$. In this case, $\ell(v)=\psi(0,v)$ is the equator of equation $z=0$. Here, $a(u)=\cos{u}$ and, therefore, Eq. \eqref{k-c} of a spherical $\alpha$-catenaries becomes 
$$
 \kappa=\frac{\alpha\cos{u}}{u\sqrt{\dot{u}^2+(\cos{u})^2}}.
$$
This equation coincides with that in Ref. \cite{lopezcatenaryform}.

Applying the Clairaut relation to the spherical catenaries ($\alpha=1$) gives
\[
 u\cos u\cos\theta = c.
\]
If $c=0$, then the catenary is a meridian of the sphere. Otherwise, if $c\not=0$, a catenary is never tangent to a meridian, and consequently, a catenary that is not a meridian can not pass through the north pole. Since the Clairaut radius $\rho=u\cos u$ of a spherical catenary has a single maximum, the intuitive picture of the spherical catenaries is that of curves oscillating around the single critical parallel. See Figure \ref{fig:GraphClairautRadius}, region A. More precisely, we have the following characterization of the qualitative behavior of spherical catenaries.

\begin{prop}
Let $\gamma(v)=\psi(u(v),v)$ be a catenary in $\mathbb{S}^2$ that is not a meridian. Then, $\gamma$ satisfies the following properties:
\begin{enumerate}
    \item there exists a unique parallel $u=u^*$ which is also a spherical catenary, where $u^*\approx 0.86$ is the unique root of $\cos u=u \sin u$ in $(0,\frac{\pi}{2})$;
    \item $\gamma$ is not asymptotic to a parallel which is not the critical parallel $u=u^*$;
    \item every catenary which is not a parallel lies in the region between two parallels $u=u_m>0$ and $u=u_M<\frac{\pi}{2}$ such that $u_m< u^* < u_M$, and
    \item all maxima of $u$ are equal to $u_M$ and all minima of $u$ are equal to $u_m$.
\end{enumerate}
\end{prop}
\begin{proof}
The function $\rho=u\cos u$ has a single maximum at $u=u^*$: $\rho'=\cos u-u\sin u$. This proves (1)

To prove (2), we note that if $\gamma$ were asymptotic to a parallel $P_0$ with $u\not=u^*$, then it would be possible to construct a sequence of length minimizing geodesics connecting two sufficiently close points of $P_0$ (length minimizing with respect to $\rmd\tilde{s}^2=u^2\rmd s^2$). Then, $P_0$ would be itself a length minimizing geodesic with respect to $\rmd\tilde{s}^2$, which would contradict (1). 

To prove (3) and (4), we proceed as follows. First, since $\gamma$ is not a meridian, it has a non-zero Clairaut constant $c\not=0$. Consequently, $u$ can not be too close to neither $u=0$ nor to $u=\frac{\pi}{2}$. In other words, there must exist $u_m>0$ and $u_M<\frac{\pi}{2}$ such that $u\in[u_m,u_M]$. Now, let $v_c$ be a zero of $\dot{u}(v)$. Then, $u(v_c)\cos u(v_c)=c$. The Clairaut radius $\rho(u)=u\cos u$, $u\in(0,\frac{\pi}{2})$, has a single maximum at $u^*$. Since $\gamma$ is not a parallel, then $c\not=u^*\cos u^*$ and, consequently, there exist precisely two values of $u$ such that $u\cos u=c$. These values are $u_m$ and $u_M$ and they must satisfy the inequality $u_m<u^*<u_M$. Finally, we also conclude that there cannot exist other minima (maxima) of $u(v)$ except $u_m$ ($u_M$, respectively).
\end{proof}

\subsection{Circular Cone}

Let $\varSigma^2$ be the cone of revolution generated by the straight-line $c(u)=(u\sqrt{2}/2,0,u\sqrt{2}/2)$. (Note that no parallel of a cone is a geodesic.) Here, $a(u)=u/\sqrt{2}$ and $\rho = \frac{1}{\sqrt{2}}u^{\alpha+1}$. Thus, no parallel of a cone is an $\alpha$-catenary. The catenary equation \eqref{kc5} becomes 
$$
\alpha\left(\dot{u}^2+\frac{u^2}{2}\right)=u\ddot{u}-2\dot{u}^2-\frac{u^2}{2}.
$$
If $\alpha=1$, then $u\ddot{u}=3\dot{u}^2+u^2$. The solution of this equation is 
$$u(v)=\frac{\mu}{\sqrt{\cos(\sqrt{2}\,v+\nu)}},\quad \mu,\nu\in\mathbb{R}.$$

\subsection{Catenoid}

Let $\varSigma^2$ be the catenoid generated by the revolution around the $z$-axis of the Euclidean catenary $z\mapsto (\cosh(z),0,z)$. We now parametrize this curve by arc-length, obtaining $c(u)= (\sqrt{1+u^2},0,\mbox{arcsinh}(u))$. The metric of the catenoid is given by $\rmd u^2+(1+u^2)\rmd v^2$, and the curve $\ell$ is the waist circle of the catenoid. Here, $a(u)=\sqrt{1+u^2}$ and $\rho=u^{\alpha}\sqrt{1+u^2}$. Therefore, no parallel of the catenoid is an $\alpha$-catenary. The equation of the catenaries \eqref{kc5} becomes 
$$
u(1+u^2)\ddot{u}-[2u^2+\alpha(1+u^2)]\dot{u}^2-(1+2\alpha)u^2-(1+\alpha)u^4-\alpha=0.
$$
In the particular case $\alpha=1$, this equation is
\begin{equation}
u(1+u^2) \ddot{u}-(1+3 u^2) \dot{u}^2-2 u^4-3u^2-1=0.    
\end{equation}

In the catenoid, Eq. \eqref{SolutionByQuadratureSurfRev} with $\alpha=1$ becomes
\[
 v-v_0 =  \pm \int_{u_0}^u \frac{c\,\rmd t}{\sqrt{1+t^2}\sqrt{t^2(1+t^2)-c^2}} = \pm c \, I(u_0,u).
\]
We will show that $I(u_0,u)$ is finite for any values of $u_0$ and $u$. Therefore, $v$ must be contained in a finite interval. This fact implies that catenaries on a catenoid always diverge. More precisely, $\gamma(v)$ may rotate a few times around the catenoid's axis; after that, it is asymptotic to some meridian $v=v^*$. Indeed, let $(v_m,v_M)$ be the maximal interval where the catenary $\gamma$ is defined, where $v_m,v_M<\infty$. The limit $\lim_{v\to v_M}u(v)$ must be infinite. Otherwise, $u(v)$ would converge to some value $u_M$. If this were the case, we could then extend $\gamma(v)$ to an interval larger than $(v_m,v_M)$, which would contradict the fact that $(v_m,v_M)$ is maximal. In conclusion, every catenary on a catenoid with respect to the waist circle must blow up in finite time.

\section{Catenaries in the Gru\v{s}in plane}\label{sec5}

The Gru\v{s}in plane $\mathbb{G}_2$ is one of the simplest examples of a sub-Riemannian geometry \cite{bellaiche1996GrusinPlane,mm}. 
 Historically, the Gru\v{s}in plane appeared in the works of Gru\v{s}in \cite{gs1,gs2} on the hypoelliptic operator $\partial_x+x^2\partial_y^2$.
 The Gru\v{s}in plane is defined by $\mathbb{G}_2=(\mathbb{R}^2_{+},\rmd s^2)$, where $\mathbb{R}_{+}^2=\{(u,v)\in\mathbb{R}^2:u>0\}$ and the metric is
$$
 \rmd s^2=\rmd u^2+\frac{1}{u^2}\rmd v^2.
$$
The geodesics of $\mathbb{G}_2$ are obtained as solutions of $\ddot{u}+\dot{v}^2/u^3=0$ and $\ddot{v}-2\dot{u}\dot{v}/u=0$. Horizontal lines $\gamma(u)=(u,v_0)$ are geodesics. On the other hand, if $\dot{v}\not=0$, then $u(t) = u_0 \cos\omega(t-t_0)$, for some $t_0,u_0\in \mathbb{R}$. Finally, we can parametrize the non-vertical geodesics of $\mathbb{G}_2$ as 
\begin{equation}\label{eq::GeodesicsOfGlusin}
 \beta_{u_0,v_0}(s) = \left(u_0\cos \frac{s}{u_0},v_0 -\frac{u_0^2}{2}(\frac{s}{u_0}+\frac{1}{2}\sin \frac{2s}{u_0}) \right). 
\end{equation}
The geodesics of the Glu\v{s}in plane are depicted in Fig. \ref{fig::CatGeodG2}. 

\begin{figure}[hbtp]
    \centering
    \includegraphics[width=0.75\linewidth]{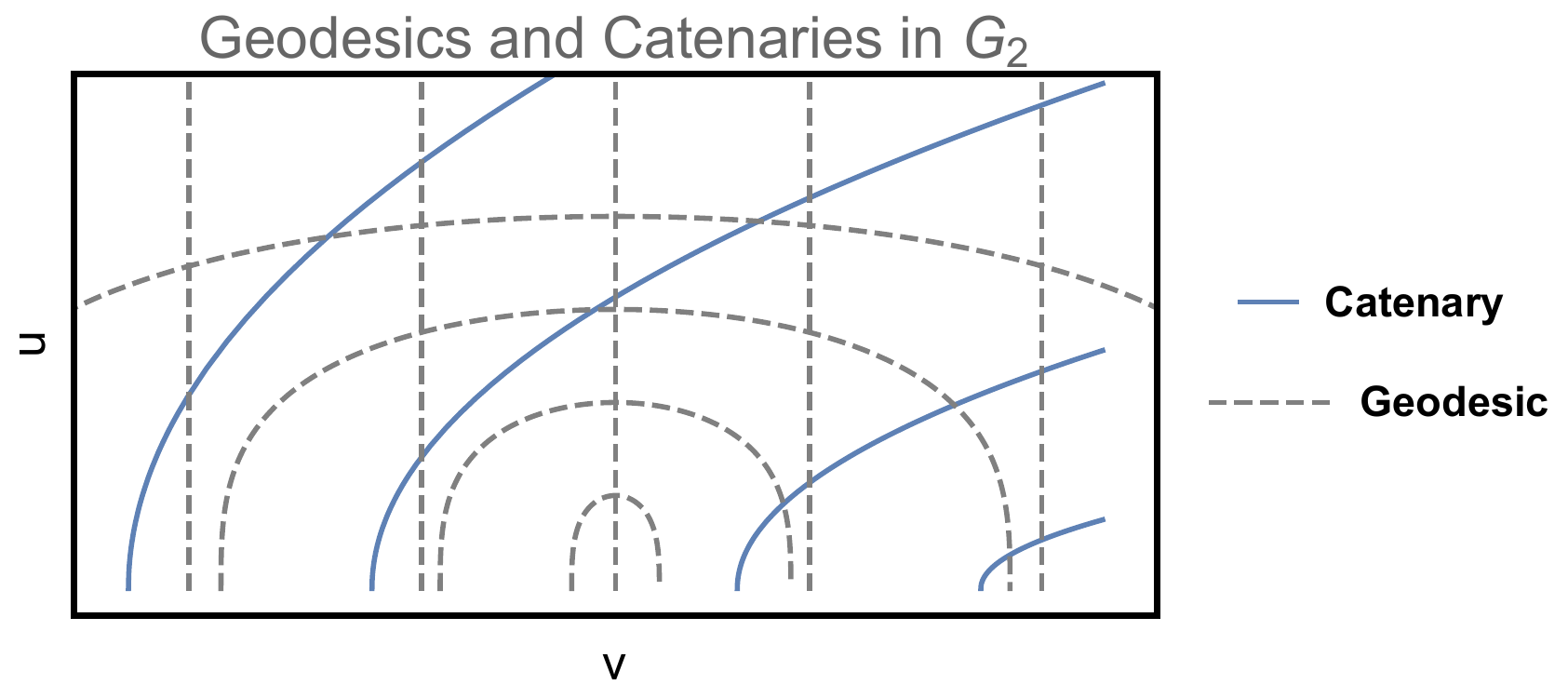}
    \caption{The catenaries and geodesics of the Glu\v{s}in plane $\mathbb{G}_2=(\mathbb{R}_+^2,\rmd s^2=\rmd u^2+\frac{1}{u^2}\rmd v^2)$. The geodesics are either lines $v=\mbox{const.}$ or the convex curves obtained in Eq. \eqref{eq::GeodesicsOfGlusin}. The catenaries are given by the graph of the square root function obtained in Eq. \eqref{eq::CatenaryOnGlusin}. The figure shows the $u$- and $v$-axes in a $1:1$ proportion.}
    \label{fig::CatGeodG2}
\end{figure}

 Fix the curve $\ell(v)=(1,v)$ as a reference line. Note that $\ell$ is not a geodesic, but it can be considered a reference line according to Definition \ref{Def::alpCatenaryRespectToNonGeod}.  The geodesic orthogonal to $\ell(v)$ at the point $(1,v)$ is the curve $\gamma(u)=(u,v)$, which implies that the coordinate system $\psi\colon\mathbb{R}^2\to\mathbb{G}_2$, $\psi(u,v)=(u,v)$, is semi-geodesic. Here, the function $G$ in Eq. \eqref{G} is $G(u,v)=1/u$. 

Let us find the $\alpha$-catenaries in $\mathbb{G}_2$. Assuming that the $\alpha$-catenary is a graph over $\ell(v)$, i.e., $u=u(v)$, the $\alpha$-catenary equation \eqref{et3} is 
 $$
      \frac{u^4}{1+u^2\dot{u}^2}\left(\frac{\ddot{u}}{u}+\frac{2\dot{u}^2}{u^2}+\frac{1}{u^4}\right)=\alpha.
 $$
In the particular case of $\alpha=1$, the equation to solve is $u\dot{u}+\dot{u}^2=0$. This equation can be integrated obtaining 
 \begin{equation}\label{eq::CatenaryOnGlusin}
     u(v)=\mu\sqrt{2 v+\nu} ,\quad \mu,\nu\in\mathbb{R}.
 \end{equation}
See Fig. \ref{fig::CatGeodG2}. Let us observe that under the change of coordinates $u=\sqrt{\bar{u}}$ and $v=\frac{1}{2}\bar{v}$, we obtain a conformal metric for the Gru\v{s}in plane: 
$$
     \rmd \bar{s}^2 = \rme^{-\ln4\bar{u}}\Big(\rmd\bar{u}^2+\rmd\bar{v}^2\Big).
$$
In this new local coordinates system, the catenaries with respect to $\ell(v)=(1,v)$ are straight lines with positive inclination, $\bar{u}= \bar{\mu}\,\bar{v}+\bar{\nu}$, $\bar{\mu}>0$.

\section*{Acknowledgements} 
Luiz da Silva acknowledges the support provided by the Mor\'a Miriam Rozen Gerber fellowship for Brazilian postdocs and the Faculty of
Physics Postdoctoral Excellence Fellowship. Rafael L\'opez  is a member of the IMAG and of the Research Group ``Problemas variacionales en geometr\'{\i}a'',  Junta de Andaluc\'{\i}a (FQM 325). This research has been partially supported by MINECO/MICINN/\break FEDER grant no. PID2020-117868GB-I00,  and by the ``Mar\'{\i}a de Maeztu'' Excellence Unit IMAG, reference CEX2020-001105- M, funded by MCINN/AEI/\break10.13039/501100011033/ CEX2020-001105-M.


\providecommand{\MR}{\relax\ifhmode\unskip\space\fi MR }
\providecommand{\MRhref}[2]{%
  \href{http://www.ams.org/mathscinet-getitem?mr=#1}{#2}
}
\providecommand{\href}[2]{#2}

\end{document}